\documentclass[sn-mathphys]{sn-jnl}

\jyear{2023}%

\theoremstyle{thmstyleone}%
\newtheorem{theorem}{Theorem}
\newtheorem{proposition}[theorem]{Proposition}%
\newtheorem{lemma}[theorem]{Lemma}%

\theoremstyle{thmstyletwo}%

\theoremstyle{thmstylethree}%

\usepackage{multirow}%
\usepackage{amsmath,amssymb,amsfonts}%
\usepackage{mathrsfs}%
\usepackage[figuresright]{rotating}%
\usepackage[title]{appendix}%
\usepackage{xcolor}%
\usepackage{textcomp}%
\usepackage{manyfoot}%
\usepackage{booktabs}%
\usepackage{algorithm}%
\usepackage{algorithmicx}%
\usepackage{algpseudocode}%
\usepackage{program}%
\usepackage{listings}%

\begin{document}

\title[On the Computation of $\log K_\nu$]{On the Computation of  the Logarithm of the Modified Bessel Function of the Second Kind}


\author{\fnm{Remi} \sur{Cuingnet}}\email{remi.cuingnet@veolia.com}


\affil{\orgname{Veolia}, \country{France}}


\abstract{
The modified Bessel function of the second kind $K_\nu$ appears in a wide variety of applied scientific fields. While its use is greatly facilitated by an implementation in most numerical libraries, overflow issues can be encountered especially for large value of $\nu$. After giving some necessary and sufficient conditions for their occurrences, this technical note shows that they can mostly be avoided by directly computing the logarithm of $K_\nu$ thanks to a simple and stable forward recursion. A statistical examples based on the Gil-Pelaez inversion formula is given to illustrate the recursive method.
}

     

\maketitle  

\section{Introduction}
\label{sec:intro}



The \textit{modified Bessel function of the second kind} $K_\nu$~\cite{abramowitz1964handbook} appears in a large number of applied scientific fields such as in statistics~\cite{bhattacharya1967modified,ROBERT1990155,NADARAJAH2016201,gaunt2019simple} where it arises in mixtures of common distributions~\cite{bhattacharya1967modified}. For instance, it arises in the probability density functions of the \textit{Generalised hyperbolic distribution}, of the \textit{K-distribution} and of the \textit{Generalized inverse Gaussian distribution}.

This function is implemented in most numerical libraries~\cite{2020SciPy-NMeth,ISO14882} following~\cite{amos1986algorithm,amos1985subroutine,campbell1980temme}.
Its computation is based on a forward recurrence 
with the relation~\cite{abramowitz1964handbook}:
\begin{equation}
    \label{eq:bessel_recurrence}
    K_{\nu+1}(z) =K_{\nu-1}(z) + 2\frac{\nu}{z}K_{\nu}(z)
\end{equation}
where the initial terms of the recurrence are computed either using power series for small values of~$z$~\cite{temme1975numerical} or based on Miller's recurrence algorithm~\cite{olver1964error,gil2007numerical} for larger values of~$z$.

As stated in~\cite{xue2018recursive}, a direct evaluation of $K_\nu(z)$ may easily cause overflow or underflow for large order $ \lvert \nu \rvert $ and for extreme argument $z$. Indeed, 
$K_\nu(z)$ has a rapid decay as $\lvert z\rvert$ grows, which may cause floating point underflow~\cite{xue2018recursive,press10numerical}. More specifically, for a fixed order $\nu$, the asymptotic behavior at infinity is~\cite{abramowitz1964handbook,watson1995treatise}:
\begin{equation}
    \label{eq:bessel_near_infty}
    K_\nu(z)
    \underset{\|z\|\to\infty}{\sim} K_{\frac12}(z) = e^{-z}\sqrt{\frac{\pi}{2z}},
\end{equation}
\noindent This is the reason why, the exponentially scaled function $\tilde{K}$ defined as
\begin{equation}
    \label{eq:exp_scaled_bessel}
    \tilde{K}_\nu(z):=e^z K_\nu(z)
\end{equation}
is often  directly implemented in most numerical packages~\cite{amos1986algorithm,press10numerical}.

Numerical issues can also be encountered for small values of $\lvert z\rvert$ especially for large orders $\nu$. As a matter of fact, for fixed $\nu$ ($\Re({\nu}) > 0$),  $K_\nu$ verifies (equations 9.6.8 and 9.6.9 in~\cite{abramowitz1964handbook}): 
\begin{equation}
    \label{eq:bessel_near_0}
    K_{0}(z)
    \underset{z \to 0}{\sim} -\log z
    \quad\textrm{and}\quad
    \forall \nu >0,\ 
    K_{\nu}(z)
    \underset{z \to 0}{\sim}
    \frac {\Gamma (\nu )}{2}
    \left({\dfrac {2}{z}}\right)^{\nu}
\end{equation}
To avoid these numerical issues, Xue and Deng~\cite{xue2018recursive} proposed a direct evaluation of  logarithmic derivatives, ratios  and  products of the modified Bessel function of the second kind. This note extends their ideas and proposes a simple direct implementation of $\log K_\nu(z)$ using a forward recursion.  

 %
 
 After giving necessary and sufficient conditions to avoid overflow and underflow for $K_\nu(z)$ in section~\ref{sec:under_over}, the direct implementation of $\log K_\nu(z)$ is detailed in section~\ref{sec:log}. The numerical stability of the recursion is assessed. The section further details how this limits the overflow and underflow problems. A numerical illustration is then given in section~\ref{sec:illustration}, where $\log K_\nu$ is used to compute the characteristic function of Student's $t$-distribution.

\section{Overflow and underflow of $K_\nu(z)$}
\label{sec:under_over}

This section focus on sufficient conditions and necessary conditions for $K_\nu(z)$
to be within the range of positive normal floating-point numbers (section~\ref{subsec:float}) based on  Yang and Chu's inequalities~\cite{yang2021monotonicity}. Section~\ref{sec:overflow} presents the results for overflow while the results for underflow are presented in section~\ref{sec:underflow}.

\subsection{Floating-point arithmetic.}
\label{subsec:float}
Given a floating-point system, let $P$ be the precision of the significand (in base 2). Let $L$ and $U$ be  the smallest and largest exponents respectively. Then, the smallest positive normalized floating-point number is the underflow level $B_\textrm{UFL} = 2^L$. The smallest positive denormalized number is $B_\textrm{SDN} = 2^{L-P}$.
The largest floating-point number is the overflow level $B_\textrm{OFL} = \left(1-2^{-P}\right)2^{U+1}$. The values for single-precision and double-precision floating point arithmetic~\cite{kahan1996} are reported in table~\ref{tab:float}.


\begin{table}[!ht]
    \centering
    \caption{Smallest positive denormalized number $B_\textrm{SDN}$, underflow level $B_\textrm{UFL}$ and overflow level $B_\textrm{OFL}$ function of the significant precision $P$, smallest exponent $L$ an largest exponent $U$ for single-precision and double precision floating point arithmetic~\cite{kahan1996}.}
    \begin{tabular}{cccc}
    \toprule
    \textbf{Name}&
    \textbf{Expression}&
    \textbf{Single Precision}&
    \textbf{Double Precision}\\
    \midrule
    Significand precision & P & 23 & 52\\
    \\
    Smallest Exponent & L &-126 &-1022\\
    \\
    Largest Exponent & U & 127 &1023\\
    \\
    Smallest positive   & $B_\textrm{SDN}$ &  $2^{-149}$                  & $2^{-1075}$\\
    denormalized number &  $=2^{L-P}$      &  $\approx1.401\cdot 10^{-45}$& $\approx4.941\cdot 10^{-324}$\\
    
    \\
    Underflow level & $B_\textrm{UFL}$   & $2^{-126}$                    &$2^{-1022}$\\
    &  $=2^{L}$      & $\approx1.175\cdot 10^{-38}$  &$\approx2.225\cdot 10^{-308}$\\
    \\
    Overflow level & $B_\textrm{OFL}$& $\left(1-2^{-23}\right)2^{128}$   & $\left(1-2^{-52}\right)2^{1024}$\\
    &                          $=\left(1-2^{-P}\right)2^{U+1}$      & $\approx3.403\cdot10^{38}$        & $\approx1.797\cdot10^{308}$\\
    \bottomrule
    \end{tabular}
    \label{tab:float}
\end{table}

The goal is then to get necessary conditions and sufficient conditions on $\nu$ and $z$ for $K_\nu(z)\leq B_\textrm{OFL}$, $K_\nu(z)\geq B_\textrm{UFL}$ or $K_\nu(z)\leq B_\textrm{SDN}$. To this end, Yang and Chu's inequalities \cite{yang2021monotonicity} are used.

\subsection{Yang and Chu's inequalities.}
While asymptotic analyses (equations~\eqref{eq:bessel_near_infty} and~\eqref{eq:bessel_near_0}) help getting an intuition on the numerical behavior, they do not provide with any bounds. A non-asymptotic characterization was given for real values by~\cite{miller2001completely} who proved the complete monotonicity
of 
$
    z~\mapsto~z^{\min{\left(\frac12, \nu\right)}} e^z K_\nu(z)
$
for non-negative $\nu$. More recently, \cite{yang2021monotonicity} gave the following upper and lower bounds for $z>0$ and $\nu\geq1$:
\begin{eqnarray}
    \label{eq:yang_bounds1}
    \displaystyle
    K_{\frac12}(z)\cdot\left({1+\frac{a_1}{z}}\right)^{\nu-\frac12}
    \quad<&\displaystyle
    K_\nu(z)
    &<\quad\displaystyle
    K_{\frac12}(z)\cdot\left({1+\frac{b_1}{z}}\right)^{\nu-\frac12}
    \\
    \label{eq:yang_bounds2}
    \displaystyle
    \left({1+a_2z}\right)^{\nu-\frac12}
    \quad<&\displaystyle
    \frac{2}{\Gamma(\nu)}\left(\frac z2\right)^\nu e^z 
    \cdot
    K_\nu(z)
    &<\quad\displaystyle
    \left({1+b_2z}\right)^{\nu-\frac12}
\end{eqnarray}
with
\begin{eqnarray}
    \label{eq:def_ab1}
    a_1 := \min\left\{{
    c_0,\ \frac\nu2 + \frac14
    }\right\}
    &\quad\textrm{and}\quad&
    b_1 := \max\left\{{
    c_0,\ \frac\nu2 + \frac14
    }\right\}
    \\
    \label{eq:def_ab2}
    a_2 := \frac{1}{\max\left\{{
    c_0,\ \nu-\frac12
    }\right\}}
    &\quad\textrm{and}\quad&
    b_2 := \frac{1}{\min\left\{{
    c_0,\ \nu-\frac12
    }\right\}}
\end{eqnarray}
where 
\begin{equation}
    \label{eq:def_c0}
    c_0 := 2\left({\frac{\Gamma(\nu)}{\sqrt\pi}}\right)^{\frac2{2\nu-1}}.
\end{equation}
Inequality~\eqref{eq:yang_bounds1} is the tightest when $z$ is large. On the contrary, inequality~\eqref{eq:yang_bounds2} is to be used for small value of $z$. The case $\nu\leq1$ is handled in~\cite{yang2017approximating}. In the following, these inequalities are used to derive sufficient or necessary conditions to avoid overflow or underflow.

\subsection{Arithmetic overflow.}
    \label{sec:overflow}

\paragraph{Necessary condition.} 
To avoid overflow, a necessary condition  is given by the following proposition.
\begin{proposition}
\label{prop:overflow_necessary}
For $B>0$ and $\displaystyle z>0$, if $\nu\geq1$ verifies:
\begin{equation}
    \label{eq:overflow_necessary}
         \nu\geq\displaystyle
    \frac12
    +
    \frac{\log(B) + z_0 - \frac12\log(\frac{\pi}{z_0e})}{
    W_0\left({\frac{2}{z_0e}
    \left({
    \log(B) + z_0 - \frac12\log\left({\frac{\pi}{z_0e}}\right)}\right)
    }\right)
    }
    \quad\textrm{with}\quad
    z_0 = \max(z, \frac{\pi}{B^2e})
\end{equation}

\end{proposition}
\noindent For instance, for $z=1$, 
$\nu \geq 29.8$ (resp. $\nu \geq 151.5$) results in an overflow with single-precision (resp. double-precision).
This necessary condition to avoid overflow for single-precision and double precision is plotted in figure~\ref{fig:overflow_underflow}.


\begin{proof}

Since, $K_\nu$ is decreasing, 
one may assume that $z\geq \frac{\pi}{B^2e}$. 
In practice, let us assume that $z> \frac{\pi}{B^2e})$. 
The case $z= \frac{\pi}{B^2e})$ will then be directly obtained by continuity.

\noindent
Since $a_2$ is defined in equation~\eqref{eq:def_ab2} is positive, inequality~\eqref{eq:yang_bounds2} yields:
\begin{equation}
\label{eq:start_proof_over}
K_\nu(z)>\frac{\Gamma(\nu)}2e^{-z}
\left(\frac 2z\right)^\nu    
\end{equation}

\noindent
The Gamma function can be bounded from below using Karatsuba's inequalities~\cite{karatsuba2001asymptotic}. Since $\nu>0$ the Gamma function verifies:
\begin{equation}
    \left({8\nu^3 + 4\nu^2 + \nu + \frac{1}{100}}\right)^{\frac16}
    <
    \frac{\nu\Gamma(\nu)}{\sqrt{\pi}}
    \left({\frac{e}{\nu}}\right)^\nu
    <
    \left({8\nu^3 + 4\nu^2 + \nu + \frac{1}{30}}\right)^{\frac16}
\end{equation}
Then for $\nu\geq1$,
\begin{equation}
    \label{eq:gamma_bounds}
    \sqrt{\frac{2}{e}}
    <
    \frac{\Gamma(\nu)}{\sqrt{\pi}}
    \left({\frac{e}{\nu}}\right)^{\nu-\frac12}
    <
    1
\end{equation}
Plugged into inequality~\eqref{eq:start_proof_over}, this gives:
\begin{equation}
 \label{eq:proof3_1}
K_\nu(z)
\quad>\quad
\frac12\sqrt{\frac{2\pi}{e}}\left({\frac{\nu}{e}}\right)^{\nu-\frac12}
e^{-z}
\left(\frac 2z\right)^\nu
\quad>\quad
e^{-z}
\sqrt{\frac{\pi}{ez}}
\left({\frac{2(\nu-1/2)}{ez}}\right)^{\nu-\frac12}
\end{equation}
%
%

\noindent
Since $W_0(x) < x$ for all $x>-\frac1e$, the constraint $z>\frac{\pi}{B^2e}$ yields to:
\begin{equation*}
     \log(B) + z - \frac12\log(\frac{\pi}{ze}) > 0
\end{equation*}
\noindent
Then, inequality~\eqref{eq:overflow_necessary} leads to:
\begin{eqnarray}
    \nonumber
      \frac2{ze}(\nu-1/2)&\geq&\displaystyle  
  \exp\left(
  {
    W_0\left({\frac{2}{ze}
    \left({
    \log(B) + z - \frac12\log\left({\frac{\pi}{ze}}\right)}\right)
    }\right)
    }\right)
        \\\nonumber
      \left({\frac2{ze}(\nu-1/2)}\right)\log\left({\frac2{ze}(\nu-1/2)}\right)&\geq&\displaystyle  
    \frac{2}{ze}
    \left({
    \log(B) + z - \frac12\log\left({\frac{\pi}{ze}}\right)}\right)
    \\\label{eq:proof3_2}
      \left({\frac2{ze}(\nu-1/2)}\right)^{\left({\nu-1/2}\right)}
      &\geq&\displaystyle  
    B e^z \sqrt{\frac{ze}{\pi}}
\end{eqnarray}
Combining inequalities~\eqref{eq:proof3_1} and~\eqref{eq:proof3_2} proves proposition~\ref{prop:overflow_necessary}.
\end{proof}

\paragraph{Sufficient condition.}
A sufficient condition for non-overflow is then given by the following proposition.
\begin{proposition}
\label{prop:overflow_sufficient}
For $z>0$, $B>0$ 
and $\nu\geq\displaystyle \frac{e}{2(4-e)}$, if

\begin{equation}
    \label{eq:overflow_sufficient}
    \nu \leq \frac{\log x_0}{W_0\left({\frac{2}{ze}\log x_0}\right)}
    -\frac{ze}{2}
    \quad\textrm{with}\quad
    x_0 :=\frac{B}{K_{\frac12}(z)}\left({1 + \frac{2}{ze}}\right)^{\frac{1+ze}{2}}
\end{equation}
where $W_0$ is the Lambert's function, 
then $K_\nu(z) < B$.
\end{proposition}
\noindent For instance, for $z=1$, 
$\nu \leq 27.7$ (resp. $\nu \leq 149.4$) is a sufficient condition to avoid 
overflow with single-precision (resp. double-precision). This sufficient condition to avoid overflow for single and double precision is plotted in figure~\ref{fig:overflow_underflow}.

\begin{proof}
Bounds on $c_0$ can be directly obtained from inequality~\eqref{eq:gamma_bounds}:
\begin{equation}
    \label{eq:bound_c0}
    2\frac\nu{e}
    \left({
        \frac{2}{e}
    }\right)^{\frac1{2\nu-1}}
    <
    c_0=2\left({
        \displaystyle\frac{\Gamma(\nu)}{\sqrt{\pi}}
    }\right)^{\frac1{\nu-\frac12}}
    <
    2\frac\nu{e}
\end{equation}
According to inequality~\eqref{eq:yang_bounds1},  
for $\nu\geq\displaystyle \frac{e}{2(4-e)}$,
\begin{eqnarray}
    \nonumber
    K_\nu(z) &<& K_{\frac12}(z) \left({ 1+ \frac{2\nu}{ze}
    }\right)^{(\nu-\frac12)}
    =
    \left({ 1+ \frac{2\nu}{ze}}\right)^{(\nu+\frac{ze}{2})}
    \left({ 1+ \frac{2\nu}{ze}}\right)^{-(\frac{ze}{2}+\frac12)} K_{\frac12}(z)
    \\
    \label{eq:overflow_sufficient_proof}
    K_\nu(z)  &< &
    \left({\frac{2}{ze}\left(\nu + \frac{ze}{2}\right)}\right)^{(\nu+\frac{ze}{2})}
    \left({ 1+ \frac{2}{ze}}\right)^{-(\frac{ze}{2}+\frac12)} K_{\frac12}(z)
    \quad\quad\textrm{(since $\nu\geq1$)}
\end{eqnarray}
Since by definition, $\nu$ satisfies inequality~\eqref{eq:overflow_sufficient}, 
\begin{eqnarray*}
    \frac{2}{ze}\left({\nu + \frac{ze}{2}}\right) 
    &\leq& \frac{\frac{2}{ze}\log x_0}{W_0\left(\frac{2}{ze}\log x_0\right)} = 
    \exp\left[{W_0\left(\frac{2}{ze}\log x_0\right)}\right]
    \\
    \left({\frac{2}{ze}\left({\nu + \frac{ze}{2}}\right) }\right)^{\left({\nu + \frac{ze}{2}}\right)}
    &\leq&
    x_0 = \frac{B}{K_{\frac12}(z)}\left({1 + \frac{2\nu}{ze}}\right)^{\frac{1+ze}{2}}
\end{eqnarray*}
Hence, using inequality~\eqref{eq:overflow_sufficient_proof}, $K_\nu(z) < B$.
\end{proof}

\subsection{Arithmetic underflow.}
    \label{sec:underflow}

\paragraph{Necessary condition.}
A necessary condition to avoid underflow can be then obtained with the following proposition.
\begin{proposition}
\label{prop:underflow_necessary}
For $\nu\geq1$ and $\displaystyle0<B\leq2\sqrt{\frac{\pi}{2e}}$, if $z$ verifies
\begin{equation}
    \label{eq:underflow_necessary}
    \left\{{
    \begin{array}{rcl}
    2 z&\geq& \displaystyle\log x_0 - \log\log x_0 + \frac{e}{e-1}\frac{\log\log x_0}{\log x_0}
    \quad\quad\textrm{with}\quad
    x_0:=\frac{2^{2\nu-1}\pi}{B^2}
    \\
    z&\geq&\displaystyle
    \max\left\{{
    \frac2e\nu,\ \frac\nu2+\frac14
    }\right\}
    \end{array}
    }\right.
\end{equation}
then
$K_\nu(z) < B$.
\end{proposition}
\noindent This necessary condition to avoid underflow for single-precision and double-precision is plotted in figure~\ref{fig:overflow_underflow}.
Proposition~\ref{prop:overflow_necessary} results in an
underflow for $z\geq \nu\log 2 + 87.6$ (resp. $z \geq \nu\log 2 + 709$) 
for single-precision (resp. double precision) with $z\geq2\nu/e\geq2/e$.
\begin{proof}
Let $\nu$, $z$ and $B$ verify the hypotheses of proposition~\ref{prop:underflow_necessary}. According to Hoorfar and Hassani's bounds on the Lambert's function $W_0$~\cite{hoorfar2008inequalities} (theorem~2.7), since $x_0\geq e$, then: 
$2z \geq W_0(x_0)$.
$W_0$ being increasing, this leads to 
\begin{equation}
   \label{eq:1proof_underflow_necessary}
    B \geq  2^{\nu-\frac12}K_{\frac12}(z)
\end{equation}
With the previous bounds on $c_0$ (inequalities~\eqref{eq:bound_c0}), $b_1$ satisfies:
$$\displaystyle
    b_1 \leq 
    \max\left\{
    \frac{2}{e}\nu,\ 
    \frac{\nu}{2}+\frac14    \right\}
    \leq z
$$
As a result, inequalities~\eqref{eq:1proof_underflow_necessary}
and~\eqref{eq:yang_bounds1} yield:
$$\displaystyle\quad
    B \geq   2^{\nu-\frac12}K_{\frac12}(z)
    \geq \left(1+\frac{b_1}{z}\right)^{\nu-\frac12}K_{\frac12}(z)
    \geq K_\nu(z)
$$
\end{proof}


\paragraph{Sufficient condition.}
A sufficient condition to avoid underflow comes directly from the monotonicity of $K_\nu$.
\begin{proposition}
\label{prop:underflow_sufficient}
For $\nu\geq\frac12$ and $\displaystyle0<B\leq\sqrt{\frac{\pi}{e}}$, if $z>0$ verifies
\begin{equation}
    \label{eq:underflow_sufficient}
    2 z\leq \displaystyle\log x_0 - \log\log x_0 + \frac{1}{2}\frac{\log\log x_0}{\log x_0}
    \quad\textrm{with}\quad \displaystyle x_0:=\frac{\pi}{B^2},
\end{equation}
then
$K_\nu(z) \geq K_{\frac12}(z) \geq B$.
\end{proposition}
\begin{proof}
Since $x_0\geq e$, according to~\cite{hoorfar2008inequalities}, inequality~\eqref{eq:underflow_sufficient} implies that
$z\leq \frac{1}{2}W_0(x_0)$. Function $W_0$ being increasing, 
$K_{\frac12}(z)\geq B$. 
\end{proof}
\noindent This sufficient condition to avoid underflow for single-precision and double precision is plotted in figure~\ref{fig:overflow_underflow}. Note that, 
As a result, $z \leq 85.3$ for single-precision and $z \leq 705$ for double-precision are sufficient condition to avoid underflow.

\paragraph{Exponential scaling.}
In practice, the use of the exponentially scaled function $\tilde K_\nu(z)$ is sufficient to avoid underflow.
\begin{proposition}
\label{prop:exp_scale}
Given a floating point arithmetic system with significand precision $P$ (in base 2),  $L$ the smallest exponent and  $U$ the largest exponent. If $L \leq -(U+1)/2$, then there is no underflow with $\tilde K_\nu(z)$ for $\nu\geq0$ and $z>0$.
\end{proposition}
\begin{proof}
Using the monotonicity of $\nu\mapsto K_\nu$ and the lower bounds given by~\cite{yang2017approximating}, 
since $z \leq \left(1-2^{-P}\right)2^{U+1} < 2^{U+1}$
\begin{equation*}
    \tilde K_\nu(z) > \sqrt{\frac{\pi}{2z + 1/2}} > 2^{-(U+1)/2}\geq 2^L.
\end{equation*}\end{proof}

In this section, sufficient condition and necessary conditions were given on $\nu$ and $z$ for $K_\nu(z)$ to lie within the range of positive normal floating-point numbers. These conditions are illustrated in figure~\ref{fig:overflow_underflow}.
While underflow may be avoided with exponential scaling (proposition~\ref{prop:exp_scale}), 
to avoid overflow issues, an alternative is to directly compute  the logarithm of $K_\nu(z)$ using forward recursion. This is detailed in the next section.
\begin{figure}[h!]
\centering
 \includegraphics[width=.95\textwidth, trim=11 15 10 10, clip]{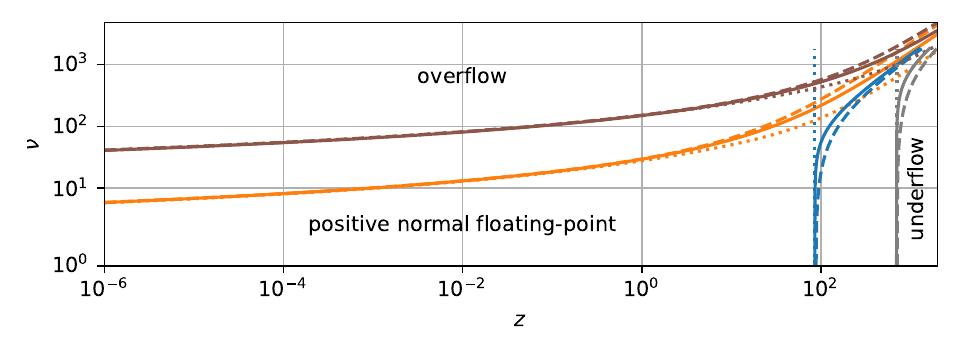}%
\caption{%
Conditions on $\nu$ and $z$ for $K_\nu(z)$ to lie with the range of positive normal floating-points. 
The conditions to avoid overflow for single-precision (resp. double-precision)
are represented in orange (resp. brown). 
The conditions to avoid underflow for single-precision (resp. double-precision)
are represented in orange (resp. gray).
Sufficient conditions 
are represented by dotted lines (inequalities~\eqref{eq:underflow_sufficient} and~\eqref{eq:overflow_sufficient}). Necessary conditions are represented by dashed lines (inequalities~\eqref{eq:underflow_necessary} and~\eqref{eq:overflow_necessary}). Solid lines correspond to numerical estimation of the frontiers by dichotomic search.}%
\label{fig:overflow_underflow}
\end{figure}

\section{{Logarithm of $K_\nu$}}
\label{sec:log}
 In this section we consider the computation of $\log K_\nu(z)$. We assume that $\nu$ and $z$ are positive real values. For $\nu<0$, one may use the fact that $K_\nu(z) = K_{-\nu}(z).$

\subsection{Recurrence Relation.}
\label{subsec:recursion}

\paragraph{Summing the logarithm of ratios.}
A straightforward way of expressing $\log K_\nu$ is by summing the logarithms of the ratios $r_\nu(z)$ defined as
\begin{equation}
    \label{eq:bessel_k_ratio}
    r_\nu(z) := \frac{K_{\nu+1}(z)}{K_{\nu}(z)} 
    =  \frac{\tilde K_{\nu+1}(z)}{\tilde K_{\nu}(z)}
\end{equation}
As a result, the logarithm of the Bessel function verifies
\begin{equation}
    \label{eq:sum_log_ratio}
    \nu-\nu_0\in\mathbb N,\ \log K_\nu = \log K_{\nu_0}(z) + \sum_{k=0}^{\nu-\nu_0-1}\log r_{\nu_0+k}
\end{equation}

The ratio $r_{\nu_0+k}$ is often evaluated using its continued fraction formulation (equations (6.6.24-27) in~\cite{press10numerical})
  \begin{equation*}
     \frac{K_{\nu+1}(z)}{K_\nu(z)}
     = 
     \frac{1}{z}\left[{
     \nu + \frac12 + x + 
     \left(\nu^2-\frac14\right)
     \frac{1}{b1+}
     \frac{a_2}{b_2+}\cdots
     }\right]
 \end{equation*}
 with 
 \begin{equation*}
     a_{n+1} = \nu^2-(n+1/2)^2 
     \quad\textrm{and}\quad
     b_n = 2(n+z),
 \end{equation*}
Another standard way of computing the ratio is the forward recursion~\cite{xue2018recursive}:
\begin{equation}
    \label{eq:ratio_rec}
    r_{\nu} = \frac{1}{r_{\nu-1}} + 2\frac{\nu}{z}
\end{equation}

When using this forward recursion,
the stability of $\log K_\nu$ as a sum of log ratio is not direct and intuitively comes from the alternating nature of the propagated error on $\log r_\nu$.
Indeed, the relative error $\epsilon_\nu$ on $r_\nu$ verifies (equation~\eqref{eq:ratio_rec}):
\begin{equation}
    \label{eq:error_rec}
    \epsilon_{\nu+1} = \frac{-\epsilon_{\nu}}{1+\epsilon_{\nu}}
    \frac{K_\nu(z)}{K_{\nu+2}(z)}
\end{equation}

As a matter of fact, on may dramatically simplify the stability analysis by noticing that summing the logarithm of the ratio when the latter is computed by forward recursion is equivalent in term of error propagation to directly computing $\log K_\nu$ with the following second-order recurrence relation.

\paragraph{Forward Recursion.}
The recurrence relation~\eqref{eq:bessel_recurrence} verified by the Bessel function $K_\nu$ can directly be adapted to its logarithm.

\noindent
Noting
\begin{equation}
    u_{\nu} := \log K_{\nu}(z),
\end{equation} 
the second-order recurrence relation~\eqref{eq:bessel_recurrence} becomes:
\begin{equation}
    \label{eq:bessel_recurrence_log}
    u_{\nu+1} = u_{\nu-1} + \log\left({1+ \frac{2\nu}{z}\exp(u_{\nu}-u_{\nu-1})
    }\right).
\end{equation}



\paragraph{Choosing $\nu_0$.}
 $\nu_0$ is chosen, small enough for $K_{\nu_0}(z)$ and $K_{\nu_0+1}(z)$ to be computed with existing libraries with existing numerical libraries~\cite{amos1985subroutine,amos1985subroutine,gil2007numerical} while ensuring numerical stability.
As detailed in the following section, numerical stability required 
$u_{\nu_0}>0$. Hence, if $u_{\nu}\leq0$, then $\log K_{\nu}(z)$ is directly computed with existing numerical libraries. Otherwise, one may numerically choose 
\begin{equation}
    \nu_0 := \nu - \left\lceil \nu-\frac12 \right\rceil + k
\end{equation}
where $k\in\mathbb N$ is the smallest integer such that $u_{\nu_0}>0$. Propositions~\ref{prop:overflow_necessary} and~\ref{prop:overflow_sufficient} reduce the search range.

To avoid underflow, the exponentially scaled function $\tilde K_\nu(z)$ (equation~\eqref{eq:exp_scaled_bessel}) is used to compute the first two terms of the recurrence.

Note that these recurrence relations hold for complex numbers by considering the congruence modulo $2\pi$ for the logarithm to be properly defined:
$$
 z \leftarrow \operatorname{Re}(z) + \left(\operatorname{Im}(z) \mod 2\pi\right)   i
$$



\subsection{Numerical Stability}
\label{subsec:stability}
In this section we consider the numerical stability of the forward recursion (equation~\eqref{eq:bessel_recurrence_log}.

Noting, $U_{\nu +1} := \left(u_{\nu+1}, u_{\nu}\right)$, the  relation~\eqref{eq:bessel_recurrence_log} can be rewritten as a first-order recurrence relation:
\begin{equation}
    U_{\nu +1} = f\left(U_{\nu}\right)
    \quad\textrm{with}\quad
        f: (x,y) \mapsto 
     y + \left({\log\left(1 + \frac{2\nu}{z}e^{x-y}\right)
    ,x}\right)
\end{equation}
The condition number for the infinity norm (e.g.~\cite{higham2002accuracy}) is then given by:
\begin{equation}
    \label{eq:condition_number}
    \kappa_\infty(x, y) := \frac{\|J(x,\ y)\|_\infty\cdot\|\mathbf (x,y)\|_\infty}{
    \|f(x,y)\|_\infty}
\end{equation}
where $J$ is the \textit{Jacobian} matrix. Considering that, in this case, its norm $\|J(x,\ y)\|_\infty$ is simply one, for positive $y$, the condition number is lower or equal to 1:
$
    \kappa_\infty(x, y) \leq 1
$.
Hence, since $f$ is a non-decreasing function, the forward recursion is stable.

As explained in~\cite{higham2002accuracy}, another risk that may affect the performance of the algorithm is the \textit{catastrophic cancellation} which occurs when subtracting nearby numbers (with respect to the machine precision). The term that might be affected in the relation~\eqref{eq:bessel_recurrence_log} is $u_{\nu}-u_{\nu-1}$ which corresponds to the log of the ratio $\frac{K_\nu}{K_{\nu-1}}$. Sufficient conditions on $\nu$ and $z$ to avoid catastrophic cancellation could be derived from inequalities~\eqref{eq:yang_bounds1} and~\eqref{eq:segura_bounds}. We do not details it in this paper. In practice this does not happen with the usual range of values.

\subsection{Sufficient conditions of normality.}
A few inequalities ensure that there is no underflow and overflow issues with the computation of $u_\nu$. First, based on theorem~1 and corollary~3 in~\cite{segura2011bounds}, for $z>0$ and $\nu\geq-\frac12$, 
\begin{equation}
    \label{eq:segura_bounds}
    \frac{\nu + \sqrt{\nu^2 + z^2}}{z}
    <
    \exp\left({ u_{\nu +1} - u_{\nu} }\right)
    \leq
    \frac{\nu+\frac12+\sqrt{(\nu+\frac12)^2+z^2}}{z}
\end{equation}
Hence, for $z>0$, if $0\leq \nu\leq\frac12\left(B_\textrm{OFL} -1 \right)\cdot z-\frac12$, there is neither overflow nor underflow for $\exp\left({ u_{\nu +1} - u_{\nu} }\right)$.

As for $u_\nu(z)$, the variable being signed, we do not consider underflow. Instead, we looked for sufficient conditions for $u\in[-B_\textrm{OFL}, B_\textrm{OFL}]$. While tight bounds can be obtained using proposition~\ref{prop:overflow_sufficient} and~\ref{prop:underflow_sufficient}, simpler bounds are sufficient. 
The fact that $u\geq -B$ can be proved using the monotonicity of $K_\nu$ and 
theorem~3.1 in \cite{yang2017approximating}:
\begin{equation}
    \forall\nu\geq0,\quad
    u_\nu(z) ~\geq~
    u_0(z) ~\geq~
    -z + \frac12\log\frac{\pi}{2z+1/2}
    ~\geq~ -2z
\end{equation}

As for the upper bound, with the monotonicity of $K_\nu$, one can assume that for $\nu \geq \frac{e^2}{2(e-2)}$, and $z\leq1$. 
\begin{lemma}
For $\nu \geq \frac{e^2}{2(e-2)}$ and $z\leq1$,
\begin{equation}
    \label{eq:lemma_bound}
     K_\nu(z)
    <
    \left(\frac{\nu}{z}\right)^{\nu}    
\end{equation}
\end{lemma}
\begin{proof}
Using equations~\eqref{eq:bound_c0} and~\eqref{eq:gamma_bounds} in \eqref{eq:yang_bounds2} yields:
\begin{equation*}
K_\nu(z) <
    \frac{\sqrt{\pi}}{2}
    \left({\frac\nu{e}}\right)^{\nu-\frac12}
    \left(\frac 2z\right)^\nu 
    \left({1+z \frac{e^2}{4\nu}
        }\right)^{\nu-\frac12}
\end{equation*}
Since $z\leq1$ and $\nu \geq \frac{e^2}{2(e-2)}$, inequality~\eqref{eq:lemma_bound} is obtained.
\end{proof}
Hence, based on~\cite{hoorfar2008inequalities}, a sufficient condition to avoid overflow of $u_\nu$ is given by the proposition~\ref{prop:u_sufficient}.
\begin{proposition}
\label{prop:u_sufficient}
Given $B>0$, for $\nu\leq0$ and $0<z<B$, if 
\begin{equation}
    \max\left\{\nu,\ \frac{e^2}{2(e-2)},\right\} \leq B +y\log\left(\frac yB\right)
    \quad\textrm{with}\quad y:=\min\{ z,\  1\}
\end{equation}
then $u(z)<B$.
\end{proposition}
\noindent
As result, a sufficient condition for $u_\nu(z)$ to avoid overflow is,  to have:
 \begin{equation}
     \nu\leq B_{\textrm{OFL}} +B_{\textrm{UFL}}\log\left(\frac{B_\textrm{UFL}}{B_\textrm{OFL}}\right), 
 \end{equation}
which is the case with existing floating-point arithmetic systems.

\subsection{Logarithm of $I_\nu$}

Computing $\log K_\nu$ might be used to evaluate the logaritm 
$\log I_\nu$ of  the \textit{modified Bessel function of the first kind} using the Wronskian identity~\cite{abramowitz1964handbook}:
\begin{equation}
    \label{eq:wronskian}
    \mathcal W\left (I_\nu(z),\ K_\nu(z)\right) = -\frac{1}{z}
    .
\end{equation}

\noindent Indeed, using the recurrence relations~\cite{abramowitz1964handbook},
\begin{eqnarray}
\displaystyle I'_\nu(z) &=& \displaystyle 
\frac{\nu}{z}I_\nu(z) + I_{\nu+1}(z) \\
\displaystyle K'_\nu(z) &=&  \displaystyle 
\frac{\nu}{z}K_\nu(z) - K_{\nu+1}(z) 
\end{eqnarray}
the Wronskian identity (equation~\eqref{eq:wronskian}) becomes:
\begin{equation}
        \label{eq:log_from_wronskian}
        \log I_\nu(z) = 
        - \log z
        -
        \log K_\nu(z)
        - \log\left({
        \frac {I_{\nu+1}(z)}{I_{\nu}(z)}+
        \frac {K_{\nu+1}(z)}{K_{\nu}(z)}
        }\right)
\end{equation}
Where the ratio $\frac {I_{\nu+1}(z)}{I_{\nu}(z)}$ may be evaluated either as a continued fraction (equation (6.6.21) 
 in\cite{press10numerical})
 \begin{equation*}
     \frac{I_{\nu+1}(z)}{I_\nu(z)}
     = 
     -\frac{\nu}{z}
     +
     \frac{1}{2(\nu+1)/z+}
     \frac{1}{2(\nu+2)/z+}
     \cdots
 \end{equation*}
 or using a backward recursion~\cite{xue2018recursive}.

 As detailed in this section, to avoid avoid overflow for $K_\nu(z)$
 a direct implementation of $\log K_\nu(z)$ based on a forward recursion might be used. The numerical stability of the recursion has been assessed. A numerical illustration is given in the next section.
 

\section{Illustration: computing the characteristic function of  Student's \textit{t}-Distribution.}
\label{sec:illustration}

In this section, $\log K_\nu$ is used to compute the characteristic function (section~\ref{subsec:gil_pelaez}) of Student's $t$-distribution  (section~\ref{subsec:student}). The computation's accuracy is indirectly assessed comparing the 
numerical evaluation of the \textit{probability density function} (\textit{pdf}) with the Gil-Pelaez formula (equation~\eqref{eq:cf_inversion} to 
evaluation with the closed-form formulation (section~\ref{subsec:illustration}).

\subsection{Characteristic functions and the Gil-Pelaez inversion formulae}
\label{subsec:gil_pelaez}
Characteristic functions are a power tool to analyze probability distributions. 
For a real-valued random variable $X$ with a probability density function $f_X$ and a cumulative distribution function $F_X$, the characteristic function $\varphi_X$ is the Hermitian function defined as:
\begin{equation}
    \begin{array}{cccl}
        \varphi_{X}: & \mathbb {R} &\to & \mathbb {C}\\
        &t&\mapsto & \displaystyle\operatorname {E} 
        \left[e^{itX}\right]
        =\int _{\mathbb {R} }e^{itx}\,dF_{X}(x)
        =\int _{\mathbb {R} }e^{itx}f_{X}(x)\,dx
    \end{array}
\end{equation}

It is especially useful to study linear combination of independent random variables. Indeed, the characteristic function of the 
sum of independent variables $X$, $Y$ verifies $\varphi_{X+Y} = \varphi_{X}\varphi_{Y}$. As for affined transformations, the characteritic function of $aX+b$ verifies
$\varphi_{aX+b}(t) = \varphi_X(at)e^{itb}$ with $(t, a, b)\in\mathbb R^3$.

Theses properties together with the inversion formula 
\begin{equation}
    \label{eq:cf_inversion}
    f_{X}(x)=\frac {1}{2\pi}\int _{\mathbf {R} }e^{-itx}\varphi _{X}(t)\,dt
    =\frac {1}{\pi}\int _{0}^{\infty} \operatorname {Re}\left[{ e^{-itx}\varphi _{X}(t)}\right]\,dt,
\end{equation}
 and a direct corollary named the Gil-Pelaez Inversion formulae, \cite{gil1951note}:
\begin{equation}
    \label{eq:cf_gil_pelaez}
    F_{X}(x)
    =\frac {1}{2}-\frac {1}{\pi }
    \int_{0}^{\infty} \frac{\operatorname{Im} \left[{e^{-itx}\varphi _{X}(t)}\right]}{t}\,dt
\end{equation}
makes it also an interesting numerical
tool~\cite{davies1973numerical,hughett1998error,witkovsky2016numerical}.

\subsection{Characteristic function of  Student's \textit{t}-Distribution}
\label{subsec:student}
The characteristic function $\phi_{\nu}$ of a Student's $t$-distribution with $\nu$ degrees of freedom has a closed-form. It is an even real-valued function verifying~\cite{hurst1995characteristic,dreier2002note,gaunt2019simple}:
\begin{equation}
\label{eq:student_characteristic}
\phi_\nu(t) =
\frac{K_{\nu/2} \left(\sqrt{\nu}|t|\right)
                    \cdot \left(\sqrt{\nu}|t| \right)^{\nu/2}}
                    {\Gamma(\nu/2)2^{\nu/2-1}} 
\end{equation}
where $K_\nu(z)$ is the modified Bessel function of the second kind and $\Gamma$ is the Gamma function~\cite{abramowitz1964handbook}.
Direct evaluation of the characteristic function $\phi_\nu(t)$  for large number of degrees of freedom $\nu\gg1$ or near zero ($t\ll1$) can cause a significant loss of precision.
As a matter of facts, equation~\eqref{eq:student_characteristic} contains a ratio between very large numbers when $\nu\gg1$ and a product between a very large number and a very small number for $|t|\ll1$ (cf. section~\ref{sec:under_over}).

To alleviate the notations, let $\psi_\nu$ be defined as,
\begin{equation}
    \label{eq:def_psi}
    \psi_\nu(z) :=
    \frac{K_{\nu}(z)}{\Gamma(\nu)}\left(\frac{z}{2} \right)^\nu,
    \ \nu >0,\ z \geq 0
\end{equation}
The characteristic function $\phi_\nu$ of a Student's distribution is then simply expressed as: 
\begin{equation}
    \label{eq:phi_from_psi}
    \phi_\nu(t) = 
    2 \psi_{\nu/2}(\sqrt\nu|t|)
\end{equation}
A standard way to handle these numerical issues is to first compute the logarithm of the different terms:
\begin{equation}
    \label{eq:psi_log}
    \psi_\nu(z) := \exp\left({
    \log K_{\nu}(z) - \log\Gamma(\nu) + \nu\log \left(\frac{z}{2} \right)
    }\right),
    \ \nu >0,\ z \geq 0
\end{equation}
where $\log K_\nu$ is computed as described in section~\ref{subsec:recursion} while the function $\log\Gamma(x)$ is already implemented in most numerical libraries~\cite{2020SciPy-NMeth,ISO14882}.

\subsection{Numerical illustration}
\label{subsec:illustration}

To assess the computation's accuracy of the 
characteristic function of the Student's $t$-distribution~
\eqref{eq:student_characteristic}, the probability density function $p$ 
was numerically  evaluated  with the Gil-Pelaez formula (equation~\eqref{eq:cf_inversion}  and compared to its closed form formulation: 
\begin{equation}
\label{eq:pdf}
p(x) = \frac{\Gamma \left(\frac{\nu+1}{2} \right)} {\sqrt{\nu\pi}\,\Gamma \left(\frac{\nu}{2} \right)} \left(1+\frac{x^2}{\nu} \right)^{-\frac{\nu+1}{2}}
\end{equation}

The Gauss–Kronrod quadrature formula implemented in the QUADPACK library~\cite{piessens2012quadpack} (with default parameters) was used for numerical integration. 

We compared the accuracy when computing the characteristic function directly using equation~\eqref{eq:def_psi} (\textit{Direct}), 
using using equation~\eqref{eq:psi_log} with the Bessel function evaluated using~\cite{amos1985subroutine} (\textit{LogDirect}) and using equation~\eqref{eq:psi_log} with the logarithm of the Bessel function evaluated as detailed in section~\ref{subsec:recursion}.
Overflows values near 0 were replaced with 1, which correspond to a first order Taylor's series approximation.
Results are presented on figures~\ref{fig:err_pdf_float} and  and~\ref{fig:err_pdf_double} for single and double precision respectively.

\section{Conclusion}
In a nutshell, this note proposed a simple and stable recurrence relation to compute the logarithm of the modified Bessel function of the second kind for real number. It detailed how this could avoid underflow and overflow. 

\begin{figure}[htbp]
    \centering
    \includegraphics[width=\textwidth, trim=20 55 20 60, clip]{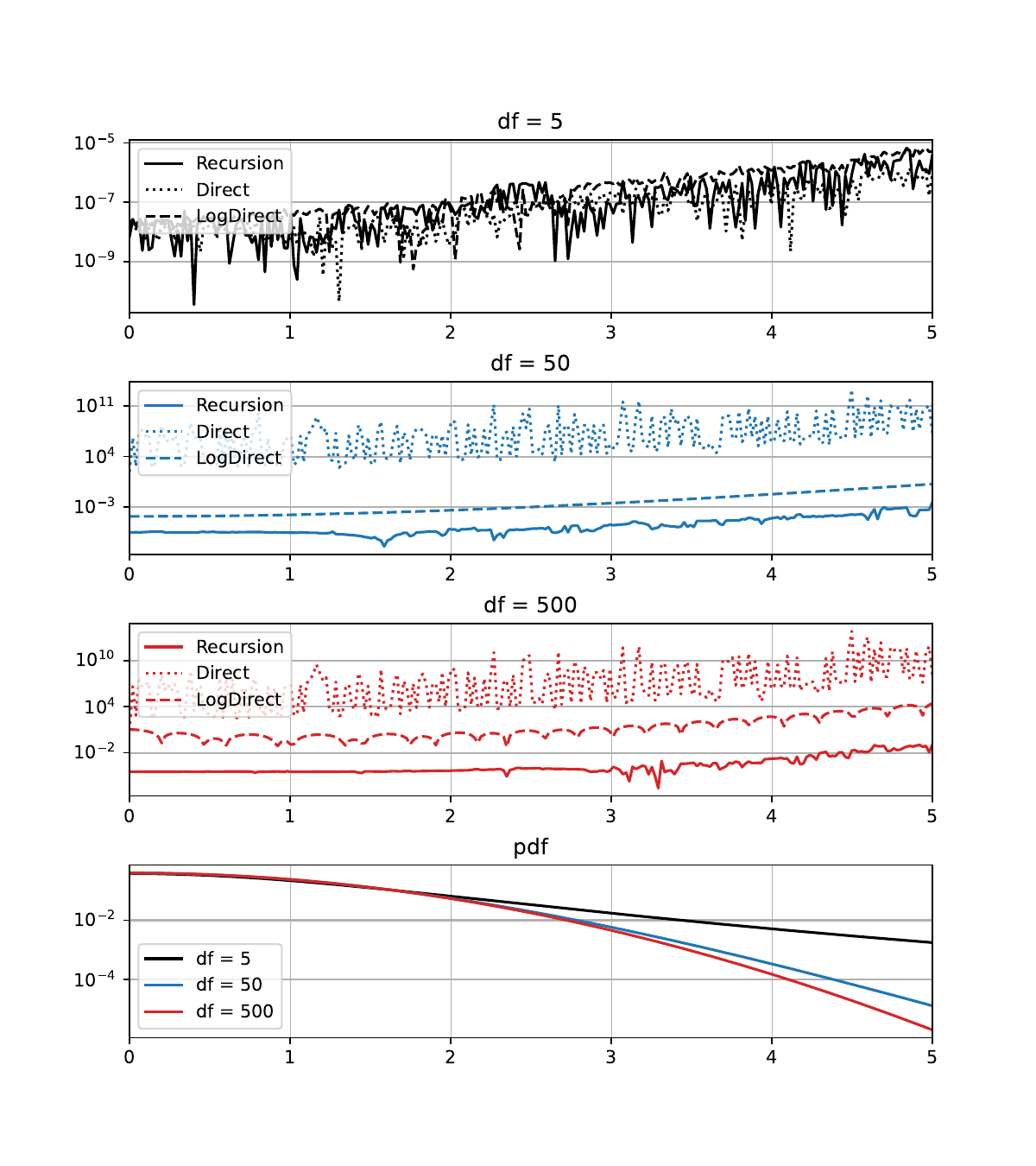}
    \caption{%
    Accuracy of the numerical evaluation of the probability density function of the Student's $t$-distribution with different degrees of freedom (df) with \textbf{single-precision}. 
    The characteristic function was computed either using equation~\eqref{eq:def_psi} (\textit{Direct}), 
or using equation~\eqref{eq:psi_log} with the Bessel function from~\cite{amos1985subroutine} (\textit{LogDirect}) or using equation~\eqref{eq:psi_log} and the logarithm of the Bessel function evaluated as detailed in section~\ref{subsec:recursion}.
Overflows values near 0 were replaced with 1, which correspond to a first order Taylor's series approximation.}
    \label{fig:err_pdf_float}
\end{figure}

\begin{figure}[htpb]
    \centering
    \includegraphics[width=\textwidth, trim=20 55 20 60, clip]{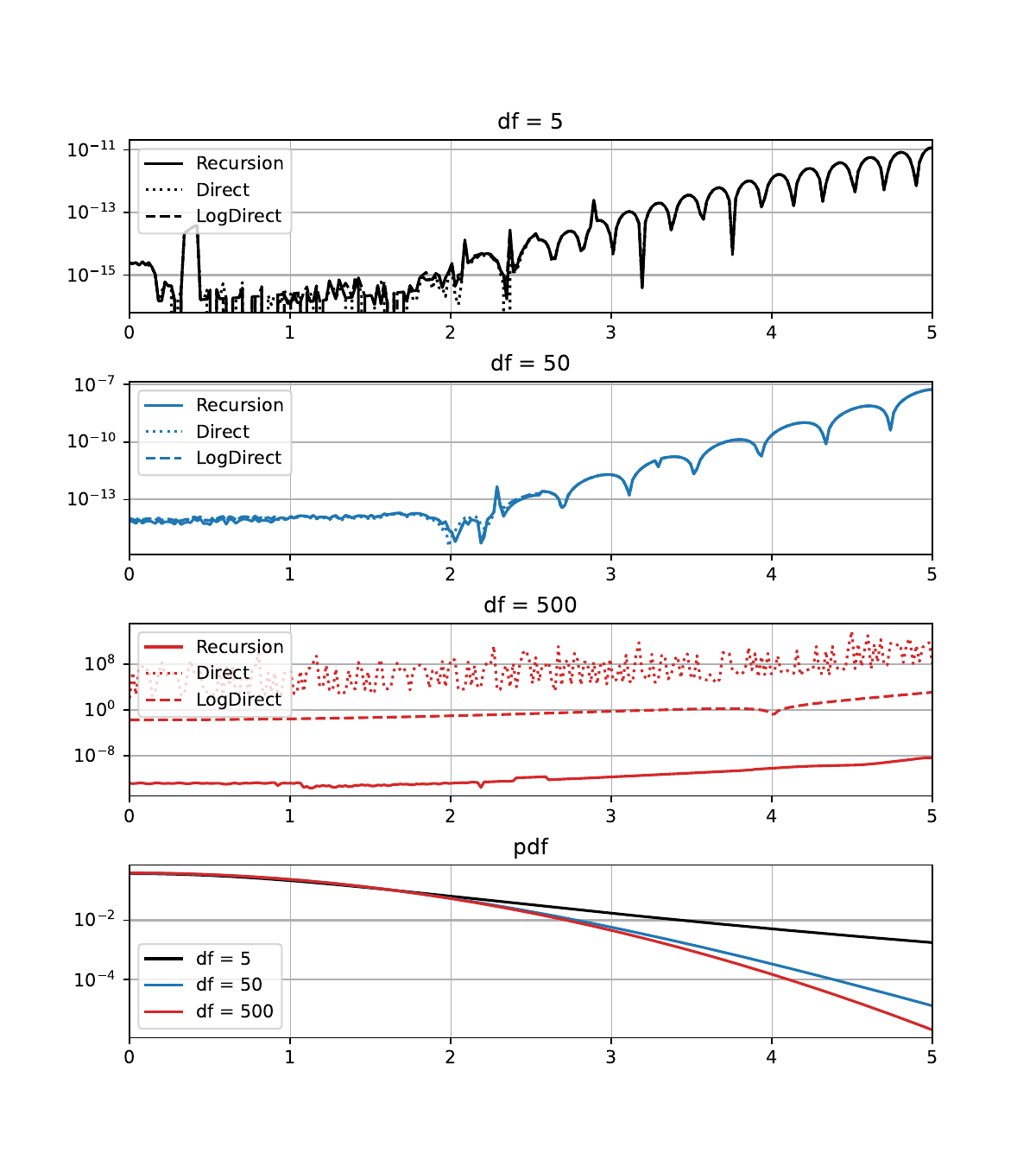}
    \caption{Accuracy of the numerical evaluation of the probability density function of the Student's $t$-distribution with different degrees of freedom (df) with \textbf{double-precision}. 
    The characteristic function was computed either using equation~\eqref{eq:def_psi} (\textit{Direct}), 
or using equation~\eqref{eq:psi_log} with the Bessel function from~\cite{amos1985subroutine} (\textit{LogDirect}) or using equation~\eqref{eq:psi_log} and the logarithm of the Bessel function evaluated as detailed in section~\ref{subsec:recursion}.
Overflows values near 0 were replaced with 1, which correspond to a first order Taylor's series approximation.}
    \label{fig:err_pdf_double}
\end{figure}



%
%


\newpage
~\newpage
\bibliography{references}
%
%

\end{document}